\documentclass[a4paper,12pt]{article}
\usepackage[latin1]{inputenc}
\usepackage{times}
\usepackage{amssymb,mathrsfs, amsmath}
\usepackage{amsmath,amsthm}
\usepackage{amssymb}
\usepackage{latexsym}
\usepackage{amsfonts}
\usepackage{mathrsfs}
\usepackage{epsfig}
\usepackage{hyperref}
\usepackage{graphicx}
\usepackage{graphics}
\newtheorem{thm}{Theorem}[section]

\newtheorem{defn}{Definition}[section]
\newtheorem{prop}{Proposition}[section]

\newtheorem{lem}{Lemma}[section]

\newtheorem{rem}{Remark}[section]

\title{Decomposition Of Invertible And Conformal Transformations}
\author{Srikanth K.V. and Raj Bhawan Yadav
 \footnote{Dept. of Mathematics, IIT Guwahati, Assam 781039, INDIA} }

\begin{document}
\maketitle{}
\begin{abstract}
In this article, we give  a geometric description for any invertible operator on a finite dimensional  inner--product space. With the aid of such a description, we are able to decompose  any given conformal transformation as a  product of planar rotations, a planar rotation or  reflection and a scalar transformation. Also,  we are able to conclude that an orthogonal transformation is a product of planar rotations and a planar rotation or a reflection.
\end{abstract}

\section{Introduction}
A classical theorem on orthogonal operators on a finite dimensional inner--product space decomposes the given operator into planar rotations and possibly a planar reflection\cite{Decm05, Kost. and Manin}. The standard proof of this theorem relies on the fundamental theorem of algebra applied to the characteristic polynomial of the  complexified  operator. Such a proof lacks geometric intuition regarding these planar rotations or reflections which arise in the decomposition of an orthogonal operator. One of the aims of this article is to provide a more geometrically inclined proof of this theorem.
In Section \ref{sec:prelim} we fix notation and recall a few elementary definitions.  Section \ref{sec:axis} introduces elementary but novel concepts such as  an \textit{axial--vector} and the \textit{axis} of a basis of a finite dimensional real inner--product space.  In Section \ref{sec:invertible} we identify a class termed   \textit{axonal} operators. Theorem
\ref{thm:decomposelinear} gives a geometric description of invertible operators using axonal operators. Finally, in Section \ref{sec:decompose}, Theorem \ref{thm:decomposeconformal}  gives a decomposition  of  conformal operators into planar rotations followed possibly by a reflection and a scalar transformation. A similar theorem for orthogonal operators is noted in Theorem \ref{thm:decomposeorthogonal}.

\section{Preliminaries}\label{sec:prelim}

Throughout this article, $V$ denotes a finite dimensional inner-- product space of dimension $n\in\mathbb{N}$ over the field of real numbers $\mathbb{R}$. The set of all invertible operators on $V$ shall be denoted by GL($V$).

\begin{defn}
Let $W$ be a two dimensional subspace of $V$. Suppose  that $\{u, v \}\subset V$ is an orthonormal basis of $W$. For any fixed real number $\theta$, the assignments \\
\hspace*{\stretch{1}}$u\mapsto \left(\cos\theta \right)u+\left(\sin\theta\right) v$ and $v\mapsto \left(-\sin\theta\right) u+\left(\cos\theta\right) v$ \hspace*{\stretch{1}}\\determine a unique linear operator on  $W$ termed as a \textbf{rotation} of $W$.

Similarly the assignments \\
\hspace*{\stretch{1}}$u\mapsto u$ and $v\mapsto -v$ \hspace*{\stretch{1}}\\determine a unique linear operator on  $W$ termed as a \textbf{reflection} of $W$.
\end{defn}
\begin{defn}
Let $T$ be a linear operator on $V$ of the form $\rho\oplus id$ where $\rho$ is an operator on a two dimensional subspace $W$ of $V$ and $id$ is the identity operator on $W^\perp$. We say that $T$ is a \textbf{planar rotation} or a \textbf{planar reflection} accordingly as $\rho$ is a rotation or a reflection of $W$. A planar rotation and a planar reflection are also termed as \textbf{rotational} and \textbf{reflectional} operators respectively.
\end{defn}
\begin{rem}
When $V$ is uni--dimensional, we allow the identity operator on $V$ to be termed a rotational operator.
\end{rem}
\section{Axis Of A Basis}\label{sec:axis}
\begin{defn}
A basis $\{u_i\}_{i=1}^n$ of the inner--product space V is
\textbf{equimodular} if  there exists a real  $\delta > 0$ such
that $\|u_i\|=\delta$,  for each $i$ in $\{1,2,\ldots , n\}$. When such a $\delta=1$, we say that the basis is \textbf{unimodular}.
\end{defn}

\begin{defn} Let $\alpha$ be a non--zero vector in an inner-- product space $V$ and $\{u_i\}_{i=1}^n$ be a basis of  $V$.  If $$\langle u_i,\alpha\rangle \langle u_j, u_j \rangle^\frac{1}{2} = \langle u_j,\alpha\rangle \langle u_i, u_i \rangle^\frac{1}{2}\;\mathrm{ for\; all  }\; i,j \in \{1,2,\ldots , n\},$$ we say that  $\alpha$ is an \textbf{axial--vector} of
the given basis $\{u_i\}_{i=1}^n$.
\end{defn}
\begin{rem}\label{lem:existaxial}
Let $\alpha$ be an axial--vector of a basis $\{u_i\}_{i=1}^n$ of an inner--product space $V$.
\begin{enumerate}
\item Clearly, the ratio $\frac{\left< u_i, \alpha \right>}{\|u_i\| \|\alpha\|}$ is independent of the choice of $i\in \{1,2,\ldots , n\}$. Thus, the axial--vector $\alpha$ makes the same angle with each of the basis vectors $u_i$ for $i \in \{1,2,\ldots , n\}$.
\item When $n=1$,   $\frac{\langle u_1,\alpha \rangle}{\|u_1\|\|\alpha\|}$ is  either $-1$ or $1$.
\item However, when $n\ge 2$,  if $\frac{\langle u_i,\alpha \rangle}{\|u_i\|\|\alpha\|}$ is either  $-1$, $0$ or $+1$, it forces $\{u_i\}_{i=1}^n$ to be linearly dependent.  In other words, the common  angle between an axial vector and the basis vectors can not be $0$,  $\frac{\pi}{2 }$ or $\pi$.
\end{enumerate}
\end{rem}
\begin{lem}
Every  equimodular basis of  a finite dimensional  inner--product space has an axial--vector. Further,
any two axial--vectors
of a given  equimodular basis  are linearly dependent.


\end{lem}
\begin{proof}
Suppose $\{u_i\}_{i=1}^n$ is an equimodular basis of $V$. Given
a non--zero real $\omega$, we prove the existence of a non--zero $\alpha \in V$ such that $\langle u_i,
\alpha\rangle = \omega$ for all $i \in \{1,2,\ldots , n\}$.  Such an $\alpha$ would be an axial--vector of the given equimodular basis.

Let $A$ be the $n\times n$ matrix  with $A_{ij}=\langle u_i,u_j\rangle $ for $i, j \in \{1,2,\ldots , n\}$.  Let $\alpha =\sum_{i=1}^nx_i u_i$ for undetermined real numbers $\{x_i\}_{i=1}^n$.
Further, let $X$ denote the column vector $(x_1,x_2,...,x_n)^\mathrm{T}$.
Then for $i \in \{1,2,\ldots , n\}$,  the collection of  $n$ equations $\langle u_i, \alpha \rangle=\langle u_i, \sum_{j=1}^nx_ju_j\rangle=\omega$ is the system $AX = \Omega$, where
$\Omega$ is the column vector $(\omega, \omega, \ldots ,\omega)^T$. Existence of a solution $X$ to the latter system suffices to prove the existence of $\alpha$.  In fact, we show that for each $\Omega$ there is a unique solution $X$ by proving that $A$
is invertible.

Suppose to the contrary that $A$ is not invertible. Then, there exists a non--zero column vector  $Y=(y_1, y_2, \ldots, y_n)^\mathrm{T}$ such that $AY=0$. Set
$s=\sum_{i=1}^n y_iu_i$. Then, $s\neq 0$ and consequently $\|s\|\neq 0$. However,  $\|s\|^2=\langle \sum_{i=1}^n y_iu_i,
\sum_{j=1}^n y_ju_j\rangle
=Y^TAY =0$, a contradiction.

If $\omega$ is non--zero, it is evident from $X = A^{-1}\Omega$ that $X\neq 0$ and hence
$\alpha$ is non--zero. This proves the existence of an axial vector for the given basis.

Suppose that $\alpha$ and $\tilde{\alpha}$ are two axial--vectors  of a given equimodular basis $\{u_i\}_{i=1}^n$ of $V$.   Let $\alpha= \sum_{i=1}^n x_i u_i$ and $\tilde{\alpha}= \sum_{i=1}^n \tilde{x}_i u_i$ for some real numbers $\{x_i\}_{i=1}^n$ and $\{\tilde{x}_i\}_{i=1}^n$. Set $X=(x_1,x_2,\ldots, x_n)^T$ and $\tilde{X}=(\tilde{x}_1,\tilde{x}_2,\ldots, \tilde{x}_n)^T$. Corresponding to the axial vectors $\alpha$ and $\tilde{\alpha}$, there  are two non--zero real numbers $\omega$ and $\tilde{\omega}$ such that   $\langle u_i, \alpha\rangle = \omega$ and   $\langle u_i, \tilde{\alpha}\rangle = \tilde{\omega}$ for all $i\in\{1,2,\ldots , n\}$. These two sets of $n$ equations are the two systems $AX=\Omega$ and $A\tilde{X}=\tilde{\Omega}$ where $\Omega=(\omega,\omega,\ldots,\omega)^T$ and $\tilde{\Omega}=(\tilde{\omega},\tilde{\omega},\ldots,\tilde{\omega})^T$. Clearly  the column vectors $\Omega$ and $ \tilde{\Omega}$ are linearly dependent. Hence the
corresponding solutions $X=A^{-1}\Omega$ and $ \tilde{X}=A^{-1}\tilde{\Omega}$ are linearly dependent. We can now conclude that the axial--vectors $\alpha$ and $\tilde{\alpha}$ are linearly dependent.
\end{proof}

\begin{thm}\label{thm:existaxial}
Every basis of a finite dimensional inner--product space has an axial--vector. Any two axial--vectors
of a given basis are
linearly dependent.
\end{thm}
\begin{proof}
If $\{v_i\}_{i=1}^n$ is a basis of $V$, then the collection $u_i :=
\frac{v_i}{\|v_i\|}$ is an equimodular basis. The
latter has an axial--vector $\alpha$ by Lemma \ref{lem:existaxial}. Consequently $ \langle u_i, \alpha \rangle =\langle u_j, \alpha \rangle$ for all $i, j\in \{1,2,\ldots , n\}$. Substituting for $u_i$, we get that $ \langle v_i, \alpha \rangle \langle v_j, v_j \rangle  =\langle v_j, \alpha \rangle \langle v_i, v_i \rangle $ for all $i, j\in \{1,2,\ldots , n\}$. Thus,
$\alpha$ is an axial--vector of the given
basis $\{v_i\}_{i=1}^n$.

Further, if $\alpha$ and  $\tilde{\alpha}$ are two axial--vectors of the given basis
$\{v_i\}_{i=1}^n$, then $\alpha$ and $\tilde{\alpha}$ are
also axial--vectors of the equimodular $\{u_i\}_{i=1}^n$ and again by Lemma \ref{lem:existaxial} are linearly
dependent.
\end{proof}
Theorem \ref{thm:existaxial} ensures  the existence and uniqueness of the
$\textbf{axis}$ of a basis defined below.

\begin{defn}
Let $\alpha$ be an axial--vector of a basis $\mathcal{B}$ of  a finite dimensional inner--product space. The
$\textbf{axis}$ of $\mathcal{B}$ is defined to be the span of $\{\alpha\}$.
\end{defn}

\begin{defn}
Given a non--zero $\alpha\in V$ and a real number $\theta \in [0,\pi]$, we define the \textbf{cone} around $\alpha$ of \textbf{vertex angle} $\theta$ by
$$ \Lambda_\alpha^\theta = \left\{x\in V|\langle x,\alpha\rangle = \|x\|\|\alpha \|\cos \theta   \right\}.$$
The span of $\{\alpha \}$ shall be termed as the \textbf{axis} of the cone $\Lambda_\alpha^\theta$.
 \end{defn}

 \begin{rem}
 \begin{enumerate}
 \item When $V$ is one dimensional, the cone around any non--zero vector is empty if the angle $\theta \neq 0,\pi$.
 \item When the dimension of $V$ is at least two, every such cone $\Lambda_\alpha^\theta$ is non--empty.
 \item From Theorem \ref{thm:existaxial}, if $\mathcal{B}$ is any basis of $V$, there exists a cone in $V$ whose vertex is the zero vector, whose axis is the axis of the basis and whose vertex angle is the common angle between each of the basis elements and an axial--vector of  $\mathcal{B}$. Such a cone is termed as the \textbf{associated cone} of basis $\mathcal{B}$.
 \end{enumerate}
 \end{rem}
\section{Geometric Description Of Invertible Operators}
\label{sec:invertible}

\begin{defn}
An invertible linear operator $T:V\rightarrow V$  is called  \textbf{axonal } if it maps
an equimodular basis $\mathcal{B}$ to an equimodular basis
$\mathcal{B}^{'}$ such that   $\mathcal{B}$ and  $\mathcal{B}^{'}$ share a common axis.
\end{defn}
We denote the set of all axonal operators on $V$ by AX($V$). The latter is a subset of GL($V$).

\begin{rem}
\begin{enumerate}
\item It would be of interest to know examples of axonal operators on $V$. Indeed, given any two equimodular bases sharing a common axis, we have an example of an axonal operator which would map one of these bases to the other.
\item Axonal linear operators which map an equimodular basis $\mathcal{B}$ to $\mathcal{B'}$ may be classified into two kinds: those for which the associated cones of $\mathcal{B}$ and $\mathcal{B'}$ are same and those for which the associated cones are distinct. Proposition~\ref{prop:axonalexample} below provides examples of axonal operators of the latter kind.
\end{enumerate}
\end{rem}
\begin{prop}\label{prop:axonalexample}
Suppose that $V$ is an inner--product space of dimension at least two. Let $\{u_i\}_{i=1}^n$ be a basis of $V$ with $\alpha$ as its axial--vector. Assume that for each $i \in \{1,2,\ldots , n\}$,  $u_i$ is rotated in the space spanned by $u_i$ and  $\alpha$ to get $v_i$   such that each $v_i$ makes the same angle $\phi$ with $\alpha$. If $\phi \not\in \{0,\frac{1}{2}\pi, \pi \}$, then $\{v_i\}_{i=1}^n$ is a basis of $V$.
\end{prop}
\begin{proof}
Without any loss of generality assume that $\{u_i\}_{i=1}^n$ is an equimodular basis. Let $S$ be the span of $\{\alpha\}$.  Since each of the $u_i's$ have the same norm and make the same angle with $\alpha$,  they have the same component, say $s$, in $S$.  For each $i\in\{1,2,\ldots , n\}$, orthogonally decompose the two collections of vectors to get
\begin{align}
u_i&=s+w_i, \text{where }s\in S \text{ and }  w_i\in S^\perp    \\
v_i&=ks+hw_i, \text{ for some real numbers } h,k \neq 0.
\end{align}
We note that if either $h$ or $k$ equals 0, angle $\phi\in\{0,\frac{1}{2}\pi, \pi\}$, contradicting hypothesis.

Next, assume a linear relation of the form $\sum_{i=1}^n \lambda_i v_i = 0$ for some real numbers $\{\lambda_i\}_{i=1}^n$. Using Eq. (2) from above, we get
\begin{align*}
k\left(\sum_{i=1}^n\lambda_i\right)s+h\left(\sum_{i=1}^n\lambda_i w_i\right)=0.
\end{align*}
The summands are in orthogonal complements $S$ and $S^\perp$ while $h, k \neq 0$. Hence we conclude
\begin{align*}
\left(\sum_{i=1}^n\lambda_i\right)s=0\quad \text{and}\quad \sum_{i=1}^n\lambda_i w_i=0.\end{align*}
Adding the two equalities, we get $\sum_{i=1}^n\lambda_i \left(s+w_i\right)=0$ and hence $\sum_{i=1}^n\lambda_i u_i=0$. From the linear independence of $\{u_i\}_{i=1}^n$, we conclude $\lambda_i=0$ for all $i\in\{1,2,\ldots , n\}$. This proves that $\{v_i\}_{i=1}^n$ are linearly independent and hence form a  basis of $V$.
\end{proof}

\begin{defn}
Let  $k$ be a  positive integer with $k\le n$. A linear transformation $S:V\rightarrow V$ is called a \textbf{$k$--shear} if there exists a $k$--dimensional subspace $W$ of $V$  such that the restriction $S_{|W^\perp}$ is the identity while $S_{|W}$ is an axonal transformation which maps an equimodular  basis $\{u_i\}_{i=1}^k$ to an equimodular  basis $\{v_i\}_{i=1}^k$ of $W$ with the same axis   so  that each $v_i$ is obtained by rotating $u_i$ by the same angle $\theta$ in the two dimensional subspace  containing $u_i$ and the axis of  $\{u_i\}_{i=1}^k$ .
\end{defn}
\begin{rem}\label{rem:axonalshear}
Any axonal transformation $A:V\rightarrow V$ can be written as $A=A'\circ S$, where $S$ is a $n$-shear and  $A'$ is an axonal transformation which maps an equimodular basis to another equimodular basis such  that the two bases have the same associated cones.
\end{rem}
\begin {thm}\label{thm:decomposelinear}
Let $T:V\rightarrow V$ be any invertible linear operator on $V$. Then, there exist a diagonal operator  $D$, an axonal operator $A$ and a rotational operator $R$ on $V$ such that $T=D\circ A\circ R$.
\end{thm}
\begin{proof}
Suppose $\{u_i\}_{i=1}^n$ is any equimodular basis of $V$.  Let $v_i = T(u_i)$ and  $w_i=\frac{v_i}{\|v_i\|}$ for all $i\in\{1,2,\ldots , n\}$. Since $T$ is invertible, $\{w_i\}_{i=1}^n$ is a unimodular basis of $V$.  Let $L_1$ be the axis of the basis $\{u_i\}_{i=1}^n$ and $L_2$ be the common axis of the bases $\{v_i\}_{i=1}^n$ and $\{w_i\}_{i=1}^n$. There exists a rotational operator $R$ of $V$ which maps $L_1$ to $L_2$. Clearly, every rotational operator is invertible and norm--preserving and hence $\{R(u_i)\}_{i=1}^n$ is a unimodular basis whose axis is $L_2$. Further the linear operator $A$  which maps  $R(u_i)$ to $w_i$ for all $i \in \{1,2,\ldots , n\}$ is axonal. Let $D$ be the diagonal operator on $V$ which maps $w_i$ to $v_i$. Since both $T$ and $D\circ A\circ R$ map the basis $\{u_i\}_{i=1}^n$ to the basis $\{v_i\}_{i=1}^n$  we can conclude that $T=D\circ A\circ R$.
\end{proof}
\begin{rem}
From Remark \ref{rem:axonalshear},  a further characterization of axonal transformations which maps an equimodular basis to another equimodular basis sharing the same associated cone would provide a finer description of invertible transformations on the lines of  Theorem \ref{thm:decomposelinear}.
\end{rem}
\section{Decomposition Of Conformal And Orthogonal Operators}\label{sec:decompose}
\begin{rem}
Recall that a linear $f:V\rightarrow V$ is said to be conformal if there exists a real $\lambda >0$ such that $\langle f(u), f(v) \rangle = \lambda  \langle u, v\rangle$ for all $u, v \in V$.
\end{rem}
\begin{rem}
A linear function  is conformal if and only if it is angle preserving.
\end{rem}
\begin{thm}\label{thm:decomposeconformal}

Given any conformal $T \in GL(V)$, we can write   $$T=D\circ \mathcal{R}\circ R_{n-2}\circ R_{n-3}\circ\dots \circ R_{2}\circ R_1,$$ where  $R_k$ is a rotational operator on $V$ for each $k\in\{1,2,\ldots, n-2\}$, $\mathcal{R}$ is either a rotational or a reflectional operator on $V$ and $D$ is a scalar operator on $V$.
\end{thm}
\begin{proof}
We induct on $n$ -- the dimension of $V$. For unidimensional $V$, every $T$ is a scalar operator. When dimension of $V$ is two, it is easily verified that every conformal $T$ is of the form $D\circ \mathcal{R}$ where $D$ is a scalar operator on $V$ and $\mathcal{R}$ is either a rotation or a reflection of $V$.

Assume next that $V$ has dimension at least three. By Theorem \ref{thm:decomposelinear}, we have $T=D\circ A_1 \circ R_1$ for linear operators $D, A_1$ and $R_1$ on $V$ such that $D$ is diagonal, $A_1$ is axonal and $R_1$ is rotational.
Suppose the axonal $A_1$ maps an equimodular basis $\{u_i\}_{i=1}^n$ to an equimodular basis $\{v_i\}_{i=1}^n$ which share a common axis. By rescaling $\{u_i\}_{i=1}^n$, we can assume that the basis $\{u_i\}_{i=1}^n$ is unimodular. By rescaling $D$, if necessary, we can assume that $\{v_i\}_{i=1}^n$ is unimodular. Assume that for some  real numbers $\lambda_i$ the diagonal operator $D$ maps $v_i$ to $\lambda_iv_i$ for each $i \in \{1,2,\ldots, n\}$.
The scalars $\{\lambda_i\}_{i=1}^n$ are non--zero as $T$ is invertible.

 Write  $A_1=D^{-1}\circ T\circ R_1^{-1}$ and in the latter composition $R_1^{-1}$  and $T$ are angle preserving. Now $T\circ R_1^{-1}$ maps $u_i$ to $\lambda_i v_i$ and hence   the angle between $u_i$ and $u_j$ equals the angle between $\lambda_i v_i$ and $\lambda_j v_j$ for $i,j\in\{1,2,\ldots, n\}$. Note that $D^{-1}$ being diagonal along $\{v_i\}_{i=1}^n$ keeps the angle between $\lambda_i v_i$ and $\lambda_j v_j$ equal to the angle between $v_i$ and $v_j$ for $i,j\in\{1,2,\ldots, n\}$. Consequently, if we take $\{u_i\}_{i=1}^n$ to be orthonormal, then $\{v_i\}_{i=1}^n$ is an orthonormal basis. We conclude that $A_1$ is orthogonal.

Now that $A_1$ and $R_1$ are conformal and $T$ is given to be conformal, we conclude that $D=T\circ R_1^{-1}\circ A_1^{-1}$ is conformal. Hence, $D$ is a scalar operator.

Suppose $\alpha$ is an axial--vector of the basis $\{u_i\}_{i=1}^n$. If the common angle between $\alpha$ and each of these basis vectors is $\theta$, then $A_1(\alpha)$ makes the same angle $\theta$ with each of the vectors from the basis $\{v_i\}_{i=1}^n$ as $A_1$ is orthogonal. Hence $A_1(\alpha)$ is an axial--vector for the basis $\{v_i\}_{i=1}^n$. However, these two bases share a common axis, say, $W$. Since $W=\mathrm{Span}\{\alpha\}=\mathrm{Span}\{A_1(\alpha)\}$ and $A_1$ is orthogonal,   $A_1(\alpha)=\pm \alpha$. By replacing $D$ by $-D$ if necessary, we assume $A_1(\alpha)=\alpha$ and hence $A_1$ is the identity on $W$. Since $A_1$ is orthogonal, $W^\perp$ is an invariant subspace of $A_1$. Thus we may write $A_1=id\oplus T_2$ where $id$ is the identity operator on $W$ and $T_2$ is an orthogonal operator on $W^\perp$ of dimension $n-1$.  We apply induction hypothesis to the orthogonal $T_2$ to realize this as a composition
$$ T_2= \tilde{D}_2\circ\tilde{\mathcal{R}}\circ \tilde{R}_{n-2}\circ \tilde{R}_{n-3}\circ\cdots \circ\tilde{R}_2,$$
where $\tilde{D}_2$ is scalar  while $\tilde{\mathcal{R}}$ is either  reflectional or  rotational and $\tilde{R}_2, \tilde{R}_3, \ldots , \tilde{R}_{n-2}$ are rotational operators on $W^\perp$. Since $T_2$ is orthogonal, $\tilde{D}_2$ is the identity on $W^\perp$.   We extend $\tilde{\mathcal{R}}, \tilde{R}_2, \tilde{R}_3, \ldots \tilde{R}_{n-2}$ respectively to $\mathcal{R}, R_2, R_3, \ldots R_{n-2}$ by declaring the latter operators to be identity on $W$. We now have\\
\hspace*{\stretch{1}} $T=D\circ \mathcal{R}\circ R_{n-2}\circ R_{n-3}\circ\dots \circ R_{2}\circ R_1$.\hspace{\stretch{1}}\qedhere
\end{proof}

\begin{thm}\label{thm:decomposeorthogonal}

Given any orthogonal $T \in GL(V)$, we can write   $$T= \mathcal{R}\circ R_{n-2}\circ R_{n-3}\circ\dots \circ R_{2}\circ R_1,$$ where  $R_k$ is a rotational operator on $V$ for each $k\in\{1,2,\ldots , n-2\}$ and $\mathcal{R}$ is either a rotational or a reflectional operator on $V$.
\end{thm}
\begin{proof}
$T$ being orthogonal is conformal. From Theorem \ref{thm:decomposeconformal}, we may write
$$T=D\circ \mathcal{R}\circ R_{n-2}\circ R_{n-3}\circ\dots \circ R_{2}\circ R_1,$$ where  $R_k$ is a rotational operator on $V$ for each $k\in\{1,2,\ldots , n-2\}$, $\mathcal{R}$ is either a rotational or a reflectional operator on $V$ and $D$ is a scalar operator on $V$. Since $T$ is orthogonal, the scalar operator $D$ has to be the identity and hence \\
\hspace*{\stretch{1}}  $T= \mathcal{R}\circ R_{n-2}\circ R_{n-3}\circ\dots \circ R_{2}\circ R_1.$\hspace{\stretch{1}}  \qedhere\\

\end{proof}

\end{document}